\documentclass[12pt]{amsart}

\usepackage{amsthm,amsmath,mathrsfs}
\usepackage{amssymb,amsthm,amsmath}
\usepackage[dvips]{geometry}
\usepackage{amscd}
\usepackage[latin1]{inputenc}

\usepackage{colortbl}
\usepackage{color,graphics}

\theoremstyle{plain}

\swapnumbers
\newtheorem{theorem}{Theorem}[section]
\newtheorem{proposition}[theorem]{Proposition}
\newtheorem{corollary}[theorem]{Corollary}
\newtheorem{lemma}[theorem]{Lemma}

\newtheorem{definition}[theorem]{Definition}

\newtheorem{remark}[theorem]{Remark}

\newcommand{\T}{\mathbb{T}}

\newcommand{\N}{\mathbb{N}}

\newcommand{\eps}{\varepsilon}

\newcommand{\SF}{{\mathcal F}}

\newcommand{\SX}{{\mathcal X}}

\newcommand{\diff}{\operatorname{Diff}}

\title{On $m$-minimal partially hyperbolic diffeomorphisms}

\author{Alexander Arbieto}
\email{arbieto@im.ufrj.br}
\address{Instituto de Matem\'atica, Universidade Federal do Rio de Janeiro.}
\author{Thiago Catalan}
\email{tcatalan@famat.ufu.br}
\address{Faculdade de Matem\'atica, Universidade Federal de Uberl\^andia.}

\author{Felipe Nobili}
\email{ felipe.nobili@gmail.com   }
\address{Universidade Federal Fluminense}

\date{\today}

\begin{document}

\maketitle

\begin{abstract} We discuss about the denseness of the strong stable and unstable manifolds of partially hyperbolic diffeomorphisms. In this sense, we introduce a concept of $m$-minimality. More precisely, we say that a partially hyperbolic diffeomorphisms is $m$-minimal if  $m$-almost every point in $M$ has its strong stable and unstable manifolds dense in $M$. We show that this property has dynamics consequences: topological and ergodic. Also, we prove the abundance of $m$-minimal partially hyperbolic diffeomorphisms in the volume preserving and symplectic scenario.

\end{abstract}

\section{Introduction}

The study of minimal foliations is a very important and classical one. By definition, all leaves are dense in the manifold for such foliations. There are a lot of examples, but perhaps the most 
simple is given by Kronecker curves, which are curves with irrational slope on the two-dimensional torus. 

Such foliations arose naturally in the theory of hyperbolic dynamics, since the invariant foliations of transitive Anosov systems are minimals. However, it turns out that the minimality brings stronger notions of transitivity. Indeed, they imply that the dynamics is topologically mixing.  Moreover, allied with some source of hyperbolicity, the minimality of the invariant foliation can 
produce robust transitivity, as Pujals and Sambarino did in \cite{PS}.

In the same spirit, the results in this paper concern about denseness of the strong stable and unstable foliations of partially hyperbolic diffeomorphisms, and the consequences of such denseness. In the following, we will define the objects and give the statements of our results.

By a foliation, we mean a partition of the manifold by submanifolds, which locally varies at least continuously (see \cite{CN}).

\begin{definition}
Let $\SF$ be a foliation and  $\SX=\{x\in M; \SF(x)$ is dense$\}$. We say that the foliation is transitive if $\SX$ is non empty. We say that $\SF$ is $\mathcal{R}$-minimal if $\SX$ is a residual subset of $M$. If $\mu$ is a Borelian probability measure then we say that $\SF$ is $\mu$-minimal if $\mu(\SX)=1$. Finally, the foliation is minimal if $\SX=M$. 
\end{definition}

There are some trivial relations between those notions. For instance, every minimal foliation is a $\mathcal{R}$-minimal and $\mu$-minimal foliation, and both implies that the foliation is transitive.   It is important to remark that a foliation is transitive if, and only if, it is $\mathcal{R}-$minimal, which is a directly consequence of the continuity of the foliation. It is easy to obtain $m$-minimal foliations (where $m$ is the Lebesgue measure) which are not minimal. Take any ergodic volume preserving Anosov diffeomorphism, and consider the orbit foliation of the suspension flow. However, it is a nice question to give sufficient conditions to  promote $m$-minimal foliations to minimal foliations. In particular, we are interested about this problem for invariant manifolds of partially hyperbolic diffeomorphisms, see below.

Now, we turn to the dynamics. Let $(M,g)$ be a compact, connected, boundaryless, Riemannian manifold. The Lebesgue measure is denoted by $m$. Any submanifold will be endowed with a metric, which is the restriction of $g$. 

Given a diffeomorphism $f$ on $M$ we say that a $Df$-invariant splitting $TM=E\oplus F$ is {\it dominated} if  there exists a positive integer $n$ such that
$$\|Df^n_x(u)\|\leq 1/2 \|Df_x^n(v)\| \textrm{  for every }x\in M \textrm{, and every } u\in E\textrm{ and } v\in F.$$ 
A diffeomorphism $f$ on $M$ is called  {\it partially hyperbolic} if there exists a continuous $Df$-invariant dominated  splitting $TM=E^s\oplus E^c\oplus E^u$, with non trivial extremal sub-bundles $E^s$ and $E^u$, and there exists $n\in \N$ such that $E^s$ and $E^u$ are uniformly contracted by $Df^n$ and $Df^{-n}$, respectively.

If the center bundle $E^c$ is trivial, then $f$ is called  {\it Anosov}. For convenience, given a partially hyperbolic diffeomorphism $f$, we consider its partially hyperbolic splitting $TM=E^s\oplus E^c\oplus E^u$, such that the extremal bundles contains all the $Df$-invariant sub-bundles of $TM$ which are contracted or expanded for some iterate of $Df$. In particular, for us, a partially hyperbolic diffeomorphism with non-trivial center bundle is not Anosov.  

By Theorem 6.1 of \cite{HPS} the strong bundles, $E^s$ and $E^u$, of a partially hyperbolic diffeomorphism $f$ is integrable. That is, there exist two strong foliations, the strong stable and strong unstable foliations, which are tangent to $E^s$ and $E^u$, respectively.  We denote these foliations by $\mathcal{F}^s$ and $ \mathcal{F}^u$, respectively.

We say that a partially hyperbolic diffeomorphism $f$ is {\it s-minimal (resp. u-minimal)} if its strong stable (resp. strong unstable) foliation $\mathcal{F}^s$ (resp. $\mathcal{F}^u$) is minimal. 

As we mentioned before, it is known that $s-$minimality and $u-$minimality implies topological properties of the dynamics. More precisely, a $s-$minimal or $u-$minimal, partially hyperbolic diffeomorphism $f$ is topologically mixing. In particular, $f$ is topologically transitive. Recall that a diffeomorphism $f$ is {\it topologically transitive} if there is a point whose forward orbit by $f$ is dense on $M$. Also, $f$ is {\it topologically mixing} if given open sets $U$ and $V$ of $M$, there exists a positive integer $n$ such that  $f^j(U)$ intersects $V$ for any $j\geq n$.

For instance, all Anosov transitive diffeomorphism are $s-$minimal and $u-$minimal. For non Anosov partially hyperbolic diffeomorphisms, it is proved by \cite{BDU} that robustly transitive partially hyperbolic diffeomorphisms with one dimensional center bundle is  either $s-$minimal or $u-$minimal, after a perturbation (See also \cite{No}). Recall that a diffemorphism is {\it robustly transitive} if every diffeomorphism sufficiently close to it, in the $C^1$-topology is transitive. 

However, for partially hyperbolic diffeomorphisms with higher center dimension nothing is known about the minimality of the strong invariant foliations.  Even so, we can say something for those weak notions of minimality.

A direct consequence of the famous  Hayashi's Connecting Lemma \cite{H}, generically any partially hyperbolic diffemorphism $f$ is such that its strong foliations $\mathcal{F}^s$ and $\mathcal{F}^u$ are $\mathcal{R}-$minimal, see also \cite{Shi}. However, we do not know an answer to the following problem: if the strong foliation is $\mathcal{R}-$minimal, is it true that $f$ is topologically transitive?

We turn our attention to explore the $m$-minimality of the strong invariant foliations.

We will denote by $\mathcal{X}^s(f)$ the set of points $x\in M$ such that $\SF^s(x)$ is 
dense. Analogously, we define $\mathcal{X}^u(f)$ using  $\SF^u$. 

\begin{definition}
Let $f$ be a partially hyperbolic diffeomorphism. We say that $f$ is {\it $ms$-minimal} if $m(\mathcal{X}^s(f))=1$. We say that $f$ is $mu$-minimal if $m(\mathcal{X}^u(f))=1$. Finally, $f$ is $m$-minimal if is both $ms$ and $mu$-minimal.
\end{definition}

In Section \ref{sectionmminimality} we will prove many  basic properties  satisfied by  $ms$ and $mu$-minimal diffeomorphisms. In particular, we prove that  $m$-minimality is a $G_{\delta}$ property in the volume preserve scenario. 

In Section 3, we prove the following  two main consequences of $m$-minimallity. These give information about the complexity of the dynamics at the topologic and ergodic level. Recall, that a diffeomorphism is \emph{weakly ergodic} if the orbit of $m$-almost every point $x$ is dense on $M$.  

\begin{theorem}
Let $f$ be a $C^1$-partially hyperbolic diffeomorphism preserving the Lebesgue measure $m$. If $f$ is  $ms$-minimal or $mu$-minimal then  $f$ is topologically mixing. Moreover, if $f$ is also $C^{1+\alpha}$ then $f$ is weakly ergodic.
 \end{theorem}

In Section 3, we also  introduce the SH property of Pujals and Sambarino, \cite{PS}, and we use it to obtain a kind of robustness of the $m$-minimality, see Proposition \ref{Shprop}.

A natural question is if the converse of the implications above holds. In Section 4, we present some examples which deals with those questions.

Our next task is to show the abundance of $ms$ and/or $mu$-minimal diffeomorphisms in two natural scenarios: volume preserving diffeomorphisms and symplectic diffeomorphisms.

We start with the volume preserving scenario. We denote by $\diff^1_m(M)$ the set formed by diffeomorphisms $f$ on $M$ that preserve the volume form $m$, i.e., $f^*m=m$. 

We say that a diffemorphism $f$ exhibit a homoclinic tangency if there exists a non empty and non transversal intersection between the stable and unstable manifold of a hyperbolic periodic point of $f$. Hence, denoting by $\mathcal{HT}$ the subset of $C^1$ diffeomorphisms exhibiting a homoclincic tangency, we  know there exists an open and dense subset in $ \diff^1_m(M)\setminus cl(\mathcal{(HT)})$ formed by partially hyperbolic diffeomorphisms. This was proved by Crovisier, Sambarino and Yang in \cite{CSY}. See also \cite{ACS}. Moreover, such diffemorphisms were such that the center bundle admits a sub splitting in one dimensional sub bundles. 

Our next result says that in the volume preserving scenario, far from homoclinic tangency, the presence of $m$-minimality is abundant.

\begin{theorem} There exists an open and dense subset $\mathcal{G}\subset \diff^1_m(M)\setminus cl\mathcal{(HT)}$, such that any $C^2-$diffeomorphism $f\in \mathcal{G}$ is a partially hyperbolic diffeomorphism which is $m$-minimal. \label{mmc2}\end{theorem}

Now, since $m$-minimality is a $G_{\delta}$ property, Proposition \ref{prop.5.1}, from the previous theorem we also have the following corollary. 

\begin{corollary} There exists a residual subset $\mathcal{R}\subset \diff^1_m(M)\setminus cl\mathcal{(HT)}$, such that any $f\in \mathcal{R}$ is a partially hyperbolic diffeomorphism which is $m$-minimal.    
\end{corollary}

We consider now, the symplectic scenario.

Let $(M,\omega)$ be a symplectic manifold, been $\omega$ the two symplectic form on $M$. In this case, $m$ will denote the volume form in $M$ induced by $\omega$. The set of symplectic diffemorphisms will be denoted by $\diff^1_{\omega}(M)$. Recall that $f$ is a symplectic diffeomorphism if $f^*\omega=\omega$. 

In this setting, the symplectic structure allow us to prove that generically any partially hyperbolic diffeomorphism is $m$-minimal. It is worth to point out that  in the symplectic setting all non Anosov diffeomorphisms  are approximated by diffeomorphisms exhibiting homoclinic tangency, see Newhouse \cite{N} .  

\begin{theorem} Let $(M^{2d},\omega)$ be a symplectic manifold, and consider $m=\omega^d$ a volume form in $M$. There exists a residual subset $\mathcal{R}\subset \diff^1_{\omega}(M)$, such that if $f\in \mathcal{R}$ is  a partially hyperbolic diffeomorphism then $f$ is $m$-minimal.  
\label{symp.sett.}\end{theorem}

The paper is organized in the following way: In section 2, we give basic properties of $m$-minimality, and in section 3, we prove some dynamical consequences of $m$-minimality. In section 4, we give some examples about the relation  between the notions of minimality described in the introduction. Finally, in Section 5 we prove the abundance of m-minimal partially hyperbolic diffeomorphisms, i.e., we prove Theorems \ref{mmc2} and \ref{symp.sett.}.

{\bf Acknowledgements:} A.A. wants to thank FAMAT-UFU,  T.C. and F.N.  want to thank IM-UFRJ for the kind hospitality of these institutions during the preparation of this work. This work was partially supported by CNPq, CAPES, FAPERJ and FAPEMIG.

\section{Basic Properties of $m$-minimality}\label{sectionmminimality}

In this section we list some basic properties related to the $m$-minimality.

We first remark that if $\nu<<m$ and $f$ is $ms$-minimal then $f$ is $\nu s$-minimal.

We now examinate the invariance by iterations.

\begin{proposition}
Let $n>0$, $f$ is $m s$-minimal if, and only if, $f^n$ is $m s$-minimal. Let $n<0$ then $f$ is $m s$-minimal if, and only if, $f^n$ is $m u$-minimal.
\end{proposition}

\begin{proof}
We only need to check that $\SF^s(x,f)=\SF^s(x,f^n)$ if $n>0$ and $\SF^s(x,f)=\SF^u(x,f^n)$ if $n<0$.
\end{proof}

Now, we study the behaviour of the property under products.

\begin{proposition}
$f$ is $m s$-minimal if, and only if, $f\times f$ is $(m\times m) s$-minimal.
\end{proposition}

\begin{proof}
We begin noticing that if $(a,b)\in \SF^s((x,y),f\times f)$, then 
$$d((f^n(x),f^n(y)),(f^n(a),f^n(b)))\to 0 \textrm{ as }n\to \infty.$$
So, this implies that $a\in \SF^s(x)$ and $b\in \SF^s(y)$. 

Now, we show that $\mathcal{X}^s(f\times f)\subset \mathcal{X}^s(f)\times \mathcal{X}^s(f)$. Indeed, let $(x,y)\in \mathcal{X}^s(f\times f)$. Since $\SF^s((x,y),f\times f)\cap (U\times V)\neq \emptyset$, for any open sets $U$ and $V$, we have that $x\in \mathcal{X}^s(f)$ and $y\in \mathcal{X}^s(f)$.  

Reciprocally, let $(x,y)\in \mathcal{X}^s(f)\times \mathcal{X}^s(f)$. Any open set of $M\times M$ contains an open set like $U\times V$. Hence, there exists $a\in U\cap \SF^s(x,f)$ and $b\in V\cap \SF^s(y,f)$. 

For any $\eps>0$, there exists $N$ such that $n\geq N$ implies $d(f^n(x),f^n(a))<\eps$ and $d(f^n(y),f^n(b))<\eps$. So 
$$d((f\times f)^n((x,y)),(f\times f)^n((a,b)))<\eps \textrm{ if }n\geq N.$$ 
Hence, $(a,b)\in \SF^s((x,y),f\times f)$.

Thus $\mathcal{X}^s(f\times f)=\mathcal{X}^s(f)\times \mathcal{X}^s(f)$ and the result follows.
\end{proof}

It is possible to preserve $m$-minimality under some conjugacies.

\begin{proposition}
Let $f$ and $g$ be two partially hyperbolic diffeomorphisms. Let $h$ be a $m$-regular homeomorphism (i.e. preserves sets of null $m$-measure), such that $h(\SF^s(x))=\SF^s(h(x))$. Then $f$ is $m s$-minimal if, and only if, $g$ is $m s$-minimal.
\end{proposition}

\begin{proof}
Just notice that $h(\mathcal{X}^s(f))=\mathcal{X}^s(g)$.
\end{proof}

\begin{remark}
About the hypothesis of the previous result. There are partially hyperbolic system which are conjugated but the conjugacy do not sends strong leaves on strong leaves, see for instance \cite{YGZ}. Even so, it is a question if both can be $ms$-minimal. Another related question is the following if $f$ and $g$ are two partially hyperbolic diffeomorphisms central conjugated (see \cite{Ha}) such that $f$ is $ms$-minimal. Is it true that $g$ is also $ms$-minimal?
\end{remark}

Now, we will show that $m$-minimality is a $G_{\delta}$ property, for any borelian probability $m$. We denote by $\mathcal{PH}^1_{m}(M)$ the set of partially hyperbolic diffeomorphisms which preserve $m$, i.e. $f^*m=m$, endowed with the $C^1$ topology.  

For any partially hyperbolic diffeomorphism $f$ we denote by  $\mathcal{X}_{\delta}^s(f)$ (resp. $\mathcal{X}_{\delta}^u(f)$ ) the subset of points $x$ in $M$ such that $\mathcal{F}^s(x)$ (resp. $\mathcal{F}^u(x)$) is $\delta$-dense in $M$. Recall,  a subset $A\subset M$ is {\it $\delta$-dense}, if $A$ intersects any open ball with diameter larger than $\delta$.  Also, we denote by $\mathcal{F}_K^{s(u)}(x)$ a compact disc with radius $K$ and centre $x$ inside the leaf $\mathcal{F}^{s(u)}(x)$.  Recall, that $\mathcal{F}_K^{s(u)}(x)$ varies continuously with respect to the diffeomorphism $f$ and with respect to $x$. In particular, we can conclude that $\mathcal{X}_{\delta}^s(f)$ (resp. $\mathcal{X}_{\delta}^u(f)$) is an open subset of $M$.

\begin{proposition}
The set of $m$-minimal diffeomorphisms is a countable intersection of open sets of $\mathcal{PH}^1_{m}(M)$.
\label{prop.5.1}\end{proposition}

\begin{proof}
For any $\eps, \delta>0$  we define:

\begin{equation}
\mathcal{B}_{m}^{s}(\eps,\delta)=\{f\in \mathcal{PH}_{m}^1(M) / \quad m(\mathcal{X}_{\delta}^s(f))>1-\eps \} \text{ and }
\label{d.B.s}\end{equation}
\begin{equation}
\mathcal{B}_{m}^u(\eps,\delta)=\{f\in \mathcal{PH}_{m}^1(M) / \quad m(\mathcal{X}_{\delta}^u(f))>1-\eps \} .
\label{d.B.u}\end{equation}

Observe that if a partially hyperbolic diffemorphism $f$ is $m s-$mini\-mal  (resp. $m u-$minimal), then $f$ belongs to $\mathcal{B}^s(\eps,\delta)$ (resp. $\mathcal{B}^u(\eps,\delta)$) for every $\eps$ and $\delta$ positive. In particular, the rest of the proof is a directly  consequence of the next lemma, Lemma \ref{l.aberto}, which implies  $m$-minimality is a $G_{\delta}$ property. 

\end{proof}

\begin{lemma}
\label{l.aberto}
The subsets $\mathcal{B}^s(\eps,\delta)$ and $\mathcal{B}^u(\eps,\delta)$ are open subsets of  $\mathcal{PH}^1_{m}(M)$. 
\end{lemma}

\begin{proof}[Proof of Lemma \ref{l.aberto}]
Since $f$ belongs to $\mathcal{B}^s(\eps,\delta)$ if, and only if, $f^{-1}$ belongs to $\mathcal{B}^u(\eps,\delta)$, it is enough to prove that $\mathcal{B}^s(\eps,\delta)$ is an open subset of $\mathcal{PH}^1_{m}(M)$. 

 Let $f\in  \mathcal{B}^s(\eps,\delta)$ be a partially hyperbolic diffeomorphism with decomposition $TM=E^s\oplus E^c\oplus E^u$. 
 
By continuity of the partially hyperbolic splitting any diffeomorphism $g$ close enough to $f$ is also  partially hyperbolic.   We suppose first that every diffeomorphism $g$ close to $f$ has a partially hyperbolic splitting with stable and unstable bundle dimensions equal to the dimensions of the respectively sub-bundles in the partially hyperbolic splitting of $f$.

Now,  given $x\in \mathcal{X}^s_{\delta}(f)$, there exists $K_x>0$ such that  $\mathcal{F}^s_{K_x}(f,x)$  is  $\delta$-dense in $M$, since $M$ is a compact manifold. Thus since the strong stable manifolds varies continuously with respect to the diffeomorphism in compact parts, there exists a neighborhood $\mathcal{V}_x$ of $f$ and a neighborhood $U_x$ of $x$, such that:
$$
\mathcal{F}^s_{K_x}(g,y) \text{ is } \delta-\text{dense in } M \text{ for every } y\in U_x \text{ and } g\in \mathcal{V}_x.
$$       
In particular, note that $U_x\subset \mathcal{X}^s_{\delta}(g)$ for every $g\in \mathcal{V}_x$.

These open sets $U_x$ give a natural open cover of $\mathcal{X}^s_{\delta}(f)$, and since  $m(\mathcal{X}^s_{\delta}(f))>1-\eps$,  we can use Vitalli's Theorem, to obtain $x_ 1, \ldots, x_k\in \mathcal{X}^s_{\delta}(f)$ such that 
$$
m\left(\bigcup_{i=1}^{i=n} U_{x_i}\right)>1-\eps.
$$
Hence, considering $\mathcal{V}=\cap_{1\leq i\leq n} \mathcal{V}_{x_i}$, we have that
$m(\mathcal{X}^s_{\delta}(g))>1-\eps$ for every $g\in \mathcal{V}$, which implies $\mathcal{V}\subset \mathcal{B}^s(\eps,\delta)$.

Now, if there is $g$ close to $f$  with different strong sub bundles  dimension from $f$, then we note that $g$  has a partially hyperbolic splitting $\tilde{E}^{ss}\oplus \tilde{E}^{cs}\oplus \tilde{E}^c\oplus \tilde{E}^{cu}\oplus\tilde{E}^{uu}$ of $TM$ to $g$, such that the stable (unstable) bundle  of $g$ is $\tilde{E}^{ss(uu)}\oplus \tilde{E}^{cs (cu)}$, with $\tilde{E}^{ss(uu)}(g)$ being a sub bundle close to $ E^{s(u)}(f)$.  Hence, by \cite{HPS}, there are invariant sub manifolds integrating $\tilde{E}^{ss(uu)}(g)$ contained in the strong stable (resp. unstable)  leaf of $g$  which are close to the strong manifolds of $f$, and thus the above arguments can also be used in this situation to conclude the proof.  

\end{proof}

\section{Dynamical Consequences of $m$-minimality}

In this section, we give three consequences of the $m$-minimality.

\subsection {Topological mixing}

In this subsection we prove the following:

\begin{proposition}
Let $f$ be a partially hyperbolic diffeomorphism preserving a volume form $m$. If $f$ is $ms$ or $mu$-minimal, then $f$ is topologically mixing. 
\label{a.e.implies.mixing}\end{proposition}

{\it Proof:  } Since $f$ is $ms$-minimal if, and only if, $f^{-1}$ is $mu$-minimal, without loss of generality we can suppose $f$ is $mu$-minimal to prove the proposition. 

  Let $U$ and $V$ be two arbitrary open sets of $M$. We choose $\eps>0$ and an open ball $B\subset U$  of diameter $2\eps$, such that the compact disc $\mathcal{F}^s_{\eps}(x)$ with radius $\eps$ and centre $x$ inside the strong stable leaf of $x$, is contained in $U$ for every $x\in B$.  Now, let $\delta>0$ be such that there is an open ball of diameter $\delta>0$ inside $V$. In particular, any $\delta$-dense subset of $M$  should intersect $V$. We also denote  $b=m(B)$. 

Now, using that $f$ is $mu$-minimal, i.e.  $m(\mathcal{X}^u(f))=1$, and the continuity of the strong unstable manifold, we can repeat the arguments in the proof of Lemma \ref{l.aberto} to find an open set $W\subset M$ and $K>0$ such that  $\mathcal{F}_K^{uu}(x)$ is  $\delta$-dense for every $x\in W$ and $m(W)\geq 1-b$.

Now, by the partial hyperbolicity of $f$ there exists $N_0>0$ such that for any $n\geq N_0$ and any $x\in M$ if $D\supset \mathcal{F}_{\eps}^{uu}(x)$  then $f^n(D)$ contains  $\mathcal{F}_K^{uu}(f^n(x))$. Using this information, we will prove that $f^n(U)\cap V\neq \emptyset$, for any $n\geq N_0$, which implies $f$ is topologically mixing, since $U$ and $V$ were taken arbitrary.  

 Given $n\geq N_0$, since $f$ preserves the Lebesgue measure $m$, $m(f^{-n}(W))=m(W)$ which is bigger than $1-b$. Hence, since $b=m(B)$, there exists  $x\in f^{-n}(W)\cap B$. By choice of $B$ we can consider a disk $D\subset \mathcal{F}^{uu}(x)\cap U$ with centre $x$ and radius $\eps>0$. Thus, $f^n(D)$ contains $\mathcal{F}^{uu}_K(x)$, since $n\geq N_0$. Therefore, provided that $f^n(x)\in W$, $f^n(D)$ is $\delta$-dense in $M$ which implies $f^n(D)\cap V\neq \emptyset$.
 
$\hfill\square$

\subsection{Weak Ergodicity}

In this sub-section we prove the following theorem: 

\begin{theorem}
Let $f$ be a $C^{1+\alpha}$-partially hyperbolic diffeomorphism preserving the Lebesgue measure $m$. If $f$ is  $ms$-minimal or $mu$-minimal then  $f$  is weakly ergodic.
 \end{theorem}

We will give two proofs. The first one is direct and uses ideas from Pesin \cite{P}.  The other, is more indirect, using a result due to Zhang  \cite{Z}. However, since it produces another results that can be useful, we put that proof also.

\vspace{0,1cm}
{\it First proof:}
We fix an open set $U$, by Poincar\'e's Recurrence Theorem, we have a subset $R\subset U$ with $m(U-R)=0$ formed by recurrent points. Hence, if $z\in R$ and $w\in \SF^s(z)$ there exists $n_k\to \infty$ such that $f^{n_k}(w)\in U$.

For any $x\in \mathcal{X}^s$, we know that there exists $y\in \SF^s(x)\cap U$. Moreover, there exists an open set $V\subset U$ containing $y$ such that $\bigcup_{z\in V}\SF^s(z)$ is a neighborhood of $x$.

But, by absolute continuity, we have that $W_x=\bigcup_{z\in R\cap V}\SF^s(z)$ has full measure in $\bigcup_{z\in U}\SF^s(z)$ and the orbit of every point in $W$ meets $U$.

Hence, using Lebesgue density points, we have that $W=\bigcup_{x\in \mathcal{X}^s}W_x$ is a full measure set. Moreover, the orbit of every point in $W$ meets $U$. Using a countable basis of neighborhoods we obtain the weak ergodicity.
$\hfill\square$

\begin{remark}
It is important to remark that the arguments used in the Proof 1 can also be used together with accessibility property to obtain weakly ergodicity, for $C^{1+\alpha}$ volume preserving partially hyperbolic diffeomorphisms. See \cite{ACW} for such proof. 
\label{argumentoACW}\end{remark}

\vspace{0,2cm}

{\it Second proof:}
The second proof is based in the following result due to Zhang. We recall that an acip is an invariant probability which is absolutely continuous with respect to Lebesgue.

\begin{theorem} [Zhang  \cite{Z}] \label{zhang} Let $f \in  \operatorname{Diff}^r(M)$ for some $r > 1$ and $\Lambda$ be a strongly partially hyperbolic set supporting some acip $\mu$. Then $\Lambda$ is bi-saturated, that is, for each point $p\in \Lambda$, the global stable manifolds and the global unstable manifolds over $p$ lies on $\Lambda$.
\end{theorem}

We remark that Zhang used this result to prove that essential accessibility implies weak ergodicity, when f supports some acip.\\

\begin{lemma} \label{L1}  Let $f$ be a $C^{1+\alpha}$, ms-minimal partially hyperbolic diffeomorphism. If $\Lambda \subset M$  is a compact $m$ invariant set with $m(\Lambda)>0$, then $\Lambda=M$.\end{lemma}

\begin{proof}  Let $\Lambda$ be a compact invariant set with positive Lebesgue measure. By Theorem \ref{zhang}, $\Lambda$ is bi-saturated. Since $f$ is ms-minimal, we have that $m(\Lambda \cap \mathcal{X}^s) = m(\Lambda) > 0$. In particular, $\Lambda \cap \mathcal{X}^s$ is non-empty. For any $x \in \Lambda \cap \mathcal{X}^s$, we have that $\mathcal{F}^s(x) \subset \Lambda$ and $cl(\mathcal{F}^s(x))=M$. Since $\Lambda$ is closed, we get that $\Lambda = M$.
\end{proof}

As a consequence we obtain a criterion to show minimality.

\begin{proposition} \label{C1} Let $f$ be a $C^{1+\alpha}$, ms-minimal partially hyperbolic diffeomorphism. If $\mathcal{X}^s$ admits some 
compact invariant subset $\Lambda$ with positive measure, then $f$ is $ms$-minimal.\end{proposition}

\begin{proof} By  Lemma \ref{L1}, we have that $\Lambda = M$, and since $\Lambda \subset \mathcal{X}^s$, we conclude that $\mathcal{X}^s = M$ , which means that $f$ is $ms$-minimal. \end{proof}

Finally, we recast weak ergodicity.

\begin{theorem} \label{T2} Every $C^{1+\alpha}$ partially hyperbolic diffeomorphism  $f$  that is $ms$-minimal is
weakly ergodic.\end{theorem}
\begin{proof}   Let $\{U_n\}_{n\in \mathbb{N}}$ be a base of the topology of $M$. For a fixed $k \in \mathbb{N}$, consider the set $A_k = \{x \in M \ | \ \mathcal{O}(x) \cap U_k = \emptyset\}$. The sets $A_k$'s are closed and $f$-invariant. Clearly, $m(A_k) < 1$, since $A_k \cap U_k = \emptyset$. By Lemma \ref{L1}, we conclude that $m(A_k) = 0$ for every $k \in \mathbb{N}$. Hence the set $\bigcap_{k \in \mathbb{N}} A^c_k$ has full measure. By construction, the orbit of every point in this set passes trough every $U_n$, so it is a dense orbit. \end{proof}

Corollary \ref{C1} and  Theorem \ref{T2} have analogous versions for the $mu$-minimal case.

\subsection{The SH property}

Pujals and Sambarino introduced in \cite{PS} a property for partially hyperbolic diffeomorphisms which they call by SH. Roughly, this property says that there are points in any unstable large disks where the dynamics $f$ behaves as a hyperbolic one.    
Moreover, they proved that SH is a robust property. An amazing consequence of such property is that implies robustness of minimality of the strong foliation. 

We could ask if SH would also imply the robustness of $m$-minimality. We do not have an answer for that, yet. What we have is the following partial result, which has the same spirit of the result of Tahzibi in \cite{T}. 

\begin{definition} (Property SH) Let $ f$ be a partial hyperbolic diffeomorphism. We say that $f$ exhibits the property SH (or has the property SH) if there exist $\lambda>1$ and $C>0$ such that for any $x \in M$ there exists $y^u(x)\in \mathcal{F}_1^{uu}(x))$ (the ball of radius 1 in $\mathcal{F}^{uu}(x)$ centered at x) satisfying
$$
m\{Df^n_{|E^c(f^l(y^u(x)))}\}>C\lambda^n \text{ for any } n>0,\; l>0.
$$
Here, in the definition, $m(.)$ is the co-norm of the linear map. 
\end{definition}

\begin{proposition} Let $f$ be a volume preserving partially hyperbolic diffeomorphism $ms$-minimal having the SH property. Then given $\eps>0$, there exists a neighborhood $\mathcal{V}$ of $f$ such that $m(\mathcal{X}^s(g))>1-\eps$ for every $g\in \mathcal{V}$. 
\label{Shprop}\end{proposition}

{\it Proof:} Since $f$ is $ms$-minimal we can choose a small neighborhood $\mathcal{V}$ of $f$ inside $\mathcal{B}(\eps,\delta)$, for $\delta>0$ arbitrary small. See (\ref{d.B.s}) in the proof of Proposition \ref{prop.5.1} to recall the definition of $\mathcal{B}(\eps,\delta)$. Moreover, we can suppose that every diffeomorphisms in $\mathcal{U}$ has SH property, by robustness of such property. Hence, if $g\in \mathcal{U}$ and $\delta>0$ is small enough, given any open set $U$, we can use property SH as in \cite{PS} to prove that  $\mathcal{F}^s(g, g^{-k}(x))$ intersects $U$ for any large positive integer $k$, and every $x\in \mathcal{X}^s_{\delta}(g)$. 

 Now, we have, by Poincaré Recurrence Theorem, that  almost every point in $\mathcal{X}^s_{\delta}(g)$ is recurrent. Also, since $\mathcal{X}^s_{\delta}(g)$ is an open set of $M$, for almost every point $x\in \mathcal{X}^s_{\delta}(g)$ there is arbitrary large positive integer $n_{k_x}$ such that $f^{n_{k_x}}(x)\in \mathcal{X}^s_{\delta}(g)$. Thus, using the first part of the proof we have $\mathcal{F}^s(x)$ must intersects $U$. Since this open set was taken arbitrary, we have just proved that almost every point $x$ in  $\mathcal{X}^s_{\delta}(g)$ also belongs to $\mathcal{X}^s(g)$. Which implies $m(\mathcal{X}^s(g))>1-\eps$, since $g\in \mathcal{B}(\eps,\delta)$.

$\hfill\square$

\section{Examples and Some Questions}

In this section, we discuss some relations between the notions introduced above.

\subsection{Weak ergodicity do not imply $m$-minimality}

Let $A$ be a linear Anosov diffeomorphism on the torus $\T^2$ and $R$ be a irrational rotation on the circle. We consider $f=A\times R$ a volume preserving partially hyperbolic diffeomorphism on $\T^3$.

Since the Lebesgue measure on $\T^2$ is mixing for $A$ and the Lebesgue measure of the circle is ergodic for $R$ we have that $f$ is ergodic with respect to the Lebesgue measure of $\T^3$. In particular, it is weakly ergodic.

However, since the irrational rotation is not topologically mixing then $f$ is not topologically mixing. Indeed, take $U$ and $V$ two open sets of the circle and consider $\T^2\times U$ and $\T^2\times V$. Our result shows then that $f$ is not $ms$ or $mu$-minimal. This can be seen directly, since the strong subbundles are tangent to $\T^2\times\{.\}$, then the strong invariant manifolds belong to these sets. Thus cannot be dense on $\T^3$.

However, a natural question is:  stronger ergodic properties of the Lebesgue measure implies $m$-minimality? We can ask for mixing measures, or even a Bernoulli measures. 

\subsection{Accessibility} 

As remarked before, it is possible that ergodicity and $m$-minimality could have some relations. However, another property which relates with ergodicity is accessibility. In fact, Pugh and Shub's conjecture (see \cite{RHRHTU}) says that accessibility implies ergodicity for $C^2$ partially hyperbolic volume preserving diffeomorphisms.

A partially hyperbolic diffeomorphism is accessible if any two points can be joined by a concatenation of curves belonging to unstable or stable manifolds. Actually, this property splits the manifold in accessibility classes. We say that a system is essentially accessible if any union of accessibility classes has zero or full measure.

We remark that the strong invariant manifolds of a linear Anosov in dimension 3 is minimal (actually the center manifold is also minimal). However, a linear Anosov is not essentially accessible, since the strong stable and unstable directions are jointly integrable, see \cite{RH} for more details on other automorphisms of the torus.

So, we have the following question. Accessibility implies $m$-minimality for partially hyperbolic diffeomorphisms?

\subsection{Foliations}

In the introduction, we mentioned that the orbit foliation of the suspension of an ergodic diffeomorphism is $m$-minimal. But it is not $\mu$-minimal for many invariant measures. Take the Dirac measure of a periodic orbit for instance. 

It is natural to study foliations which are $\mu$-minimal for a large set of measures on the manifold. We think that this study could have interest on its own.

\section{The abundance of $m$-minimallity}

\subsection{A criterion to see density of the strong leafs}\label{sub1}

In this section we obtain information about the strong stable and unstable leafs of points in the manifold containing hyperbolic periodic points in their $\omega-$limit or $\alpha-$limit sets. Recall that  the {\it $\omega-$limit (resp. $\alpha-$limit) set of $x$, $\omega(x)$ (resp. $\alpha(x)$), } is the set of points $y$ in $M$ such that there exists a sequence of forward  (resp. backward ) iterates of $x$ converging to $y$. 

What follows is our criterion to show $\delta$-density of strong leafs of a partially hyperbolic diffeomorphisms. 
In the results bellow we set only the case of strong stable leafs. However, there are similar results for the strong unstable leafs, which can be obtained by considering $f^{-1}$.

\begin{proposition}
Let $f$ be a $C^1$ partially hyperbolic diffeomorphism with splitting  $TM=E^s\oplus E^c\oplus E^u$, having a periodic point $p$ with period $\tau(p)$. Given $\delta>0$, if: 
\begin{itemize}
\item[a)] $\mathcal{F}^s(f^j(p))$ is $\delta$-dense in $M$, for any $0\leq j<\tau(p)$;

\item[b)] For any small enough neighborhood $V$ of $p$, there exist a submanifold $D\subset V$ containing $p$ which integrates $E^c\oplus E^u$, i.e. $T_DM=E^c\oplus E^u$, such that $f^{-\tau(p)}(D)\subset D$; 
\end{itemize}
Then, for any $x\in M$ such that $p\in \omega(x)$ we have $\mathcal{F}^s(x)$ is also $\delta$-dense in $M$.
\label{criterion.density}\end{proposition}

\begin{remark}
Recalling that the {\it index} of a hyperbolic periodic point is its stable bundle  dimension, we remark that every hyperbolic periodic point $p$ of a partially hyperbolic diffeomorphism $f$ having index equal to the dimension of the strong stable bundle  of $f$, satisfies  condition $(b)$ of the previous proposition. 
\label{rmk.criterion}\end{remark}

Before we prove Proposition \ref{criterion.density}, we use it and the previous remark to obtain a criterion to see density of strong stable leafs.

\begin{proposition}
Let $f$ be a $C^1$ partially hyperbolic diffeomorphism with splitting  $TM=E^s\oplus E^c\oplus E^u$,  having a hyperbolic periodic point $p$ with index $s$, where $s=dim\ E^s$, such that $\mathcal{F}^s(p)$ is dense in $M$. Thus, if $x\in M$ is such that $p\in \omega(x)$ then $\mathcal{F}^s(x)$ is dense in $M$.
\label{PPP1}\end{proposition}

{\it Proof: } Since $f$ is a diffeomorphism implies that $f^j(p)$ is dense in $M$ for any integer $j$. In particular, item (a) of Proposition \ref{criterion.density} is satisfied for any $\delta>0$. Therefore, since item (b) of such proposition is also true by  Remark \ref{rmk.criterion}, we have that $\mathcal{F}^s(x)$ is dense in $M$. 

$\hfill\square$  

Let we prove now Proposition \ref{criterion.density}.

{\it Proof of Proposition \ref{criterion.density}}
Let $p$ be the periodic point of $f$ satisfying items (a) and (b) in the hypothesis.. And let $x\in M$ different of $p$, such that $p\in \omega(x)$.  Hence, there exist positive integers $n_k$ converging to infinity when $k$ goes to infinity, such that $f^{n_k}(x)$ converges to $p$. Recalling that $\tau(p)$ is the period of $p$, it is not difficulty to see that there exists some $0\leq j<\tau(p)$ such that $n_k+j$ is a multiple of $\tau(p)$ for infinitely many positive integers $k$. Hence, replacing the sequence $(n_k)_{k\in\N}$ for a subsequence we can assume $n_k+j=m_k\tau(p)$ for any $k\in \N$.

Let $U\subset M$ be an arbitrary open set, containing a disk with diameter larger than $\delta$.

Using  item (a) of the properties satisfied for $p$, there exists $K>0$ such that the compact part $\mathcal{F}_K^s(f^j(p))$ of the strong stable leaf of $f^j(p)$ intersects $U$, for every $0\leq j<\tau(p)$. Moreover,  since the strong stable leafs varies continuously in compact parts there exists a neighborhood $V$ of $p$ such that  $\mathcal{F}_K^s(y)\cap U\neq \emptyset$ for every $y\in f^j(V)$, $0\leq j<\tau(p)$. 

Taking $V$ smaller, if necessary, let $D\subset V$ the sub manifold given by item (b) in the hypothesis. Moreover, as a consequence of the partial hyperbolicity  of $f$, we can take another small neighborhood  $\tilde{V}\subset V$ of $p$, such that  $\mathcal{F}_{loc}^s(y)$ intersects transversally $f^j(D)$ for any $y\in f^j(\tilde{V})$, $0\leq j<\tau(p)$.
In particular, by choice of $x$, there exists $k_0$ such that $\mathcal{F}^s(f^{n_{k_0}}(x))$ intersects transversally $D$ in a point $z_{k_0}$. 
Now, by choice of $D$, we have that $f^{-m_{k_0}\tau(p)}(z_{k_0})\in f^j(D)\subset f^j(V)$.  
Which implies that $\mathcal{F}^s(f^{-m_{k_0}\tau(p)}(z_{k_0}))$ intersects $U$. Thus, if we observe that by choice of $j$ we have that $x\in \mathcal{F}^s(f^{-m_{k_0}\tau(p)}(z_{k_0}))$, then we conclude $\mathcal{F}^s(x)$ intersects $U$. And since $U$ was taken arbitrary we have just finished the proof of proposition. 
 
$\hfill\square$

\subsection{The existence of dense strong leafs:}\label{sub2}

The next result is a consequence of a standard application  of the Hayashi's connecting lemma applied to transitive dynamics.

  \begin{proposition}
Let $f\in \diff^1_{m}(M)$ (resp. $f\in \diff^1_{\omega}(M)$) be  partially hyperbolic, and $p$ be a hyperbolic periodic point (resp. either hyperbolic or $m$-elliptic periodic point).  Given a small enough neighborhood $\mathcal{U}$ of $f$ in $\diff^1_m(M)$ (resp. in $ \diff^1_{\omega}(M)$), there exists a dense subset $\mathcal{D}\subset \mathcal{U}$ formed by partially hyperbolic diffeomorphisms such that $\mathcal{F}^s(p(g),g)$ is dense in $M$ for every $g\in \mathcal{D}$. Where $p(g)$ denotes the analytic continuation of $p$ for $g$.
\label{prop.density}\end{proposition}

{\it Proof:} 

We can reduce $\mathcal{U}$, if necessary, such that is defined an analytic continuation to the hyperbolic periodic point $p$. If $p$ is a $m$-elliptic periodic point, then there also exists an analytic continuation for such point in the symplectic setting. Also, in this last case it is important to remark that $2m$ is a number smaller than center dimension of $f$.  

From now on in this proof we suppose  $p$ is hyperbolic and $f$ is volume preserving. For the other cases the proof is the same. Reducing $\mathcal{U}$ again, if necessary, we can suppose every diffemorphism in $\mathcal{U}$ is partially hyperbolic. 

Let $U_1,\ldots, U_n,\ldots$ be an enumerable basis of opens sets of $M$. We define  $\mathcal{B}^s_{m}\subset \mathcal{U}$ the subset of diffeomorphisms $g$ such that $g$ is partially hyperbolic and $\mathcal{F}^s(p(g),g)$ intersects $U_m$. By continuity of the strong stable foliation $\mathcal{B}^s_{m}$ is an open set inside $\mathcal{U}$. 

Let $\mathcal{R}\subset \diff^1_m(M)$ the residual subset given by \cite{BC} formed by transitive diffeomorphisms (in the symplectic setting this is prove in \cite{ArBC}). In particular, $\mathcal{R}$ is dense in $\mathcal{U}$. Let $g\in \mathcal{R}\cap \mathcal{U}$. Since $g$ is transitive, we can use the connecting lemma ( see XW )  to perturb $g$ and find a  partially hyperbolic diffeomorphism $\tilde{g}$ arbitrary close to $g$ such that $\mathcal{F}^s(p(\tilde{g}),\tilde{g})$ intersects $U_m$. Thus, we have that $\mathcal{B}^s_m$ is open and dense in $\mathcal{U}$. Since, $U_m$ was taken arbitrary, $ \mathcal{D}=\cap \mathcal{B}^s_m$ is a dense subset in $\mathcal{U}$ formed by diffeomorphisms satisfying  the thesis of the proposition.  

$\hfill\square$

\subsection{$m$-minimality in the conservative setting:}

In this section we prove Theorem \ref{mmc2}.  Before that, although we know that weakly ergodicity does not implies m-minimality, the following result says that weakly ergodicity implies $m$-minimality in some setting. 

\begin{theorem}
Let $f\in \diff^1_m(M)$ be weakly ergodic and partially hyperbolic with decomposition $TM=E^s\oplus E^c\oplus E^u$. If there exists a hyperbolic periodic point $p$ of $f$ with $ind\ p=dim\ E^s$ (resp. $ind\ p=dim\ E^s\oplus E^c$), and such that $\mathcal{F}^s(p)$ (resp. $\mathcal{F}^u(p)$ ) is dense in $M$, then $f$ is $ms$-minimal (resp. $mu$-minimal).      
\label{mmawe}\end{theorem}

{\it Proof:} This theorem is in fact a directly consequence of our criterion in the subsection \ref{sub1}.  In fact, since $f$ is weakly ergodic, then for almost every point $x$ in $M$ the forward orbit of $x$ is dense in $M$, in particular the hyperbolic periodic point $p\in \omega(x)$, which implies by Proposition \ref{PPP1} that $\mathcal{F}^s(x)$ is dense in $M$, for almost every point $x$. 

Respectively, if $ind\ p=dim\ E^s\oplus E^c$ we can use Proposition \ref{PPP1}, as before,   but now for $f^{-1}$ to conclude that $\mathcal{F}^u(x)$ is dense for almost every point $x$.

$\hfill\square$

In the sequence we will use Theorem \ref{mmawe} to prove Theorem \ref{mmc2}. 

Another important tool we use in the proof of Theorem \ref{mmawe} is  the well known blender sets, introduced by Bonatti and Diaz, \cite{BC}.  What follows is a definition of a blender given in \cite{BDV}. 

\begin{definition}
Let $f:M\rightarrow M$ be a diffeomorphism and $p$ a hyperbolic periodic point of index $i$. We say that $f$ has a {\it blender associated to $p$} if there is a $C^1-$neighborhood $\mathcal{U}$ of $f$ and a $C^1$ open set $\mathcal{D}$ of embeddings of an $(d-i-1)-$dimensional disk $D$ into $M$, such for every $g\in \mathcal{U}$, every disk $D\in \mathcal{D}$ intersects the closure of $W^s(p(g))$, where $p(g)$ is the continuation of the periodic point $p$ for $g$. Moreover, we say that a blender is {\it acctivated} by a hyperbolic periodic point $\tilde{p}$ of index $i+1$ if the unstable manifold of $\tilde{p}$ contains a disk of the superposition region.
\end{definition}

According to the above definition we have the following result: 

\begin{lemma}[Lemma 6.12 in \cite{BDV}]
Let $f:M\rightarrow M$ be a diffeomorphism having a blender associated to a hyperbolic periodic point  $p$  of index $i$. Suppose that the blender is activated by a hyperbolic periodic point $\tilde{p}$ of index $i+1$. Then, for every diffeomorphism $g$ in a small enough $C^1-$neighborhood of $f$, the closure of $W^s(p(g))$ contains $W^s(\tilde{p}(g))$.
\label{l.livro}\end{lemma}

The next result gives an abundance of Blenders in the conservative setting, (See also \cite{Ca}).

\begin{theorem}[Theorem 1.1 in \cite{RHRHTU}]
Let $f\in \diff^r_m(M)$  such that $f$ has two hyperbolic periodic points $p$ of index $i$ and $\tilde{p}$ of index $i+1$. Then there are $C^r$ diffeomorphisms arbitrary $C^1-$close to $f$ which preserve $m$ and admit a blender associated to the analytic continuation of $p$. 
\label{rrtu}\end{theorem}

What follows is a by-product of a conservative version of results in \cite{CSY} and \cite{ABCDW}:

\begin{theorem}
There is a residual subset $\mathcal{R}\subset \diff_m^1(M)-\{cl(\mathcal{HT})\}$ such that for every $f\in \mathcal{R}$,  $f$ is partially hyperbolic having non trivial extremal sub bundles with decomposition $TM=E^s\oplus E^c\oplus E^u$,  and moreover there exists hyperbolic periodic points $p_0,\ldots, p_k$ of $f$, where $k=dim\ E^s\oplus E^c$, such that $ind\ p_i=dim\ E^s+i$, for any $i=0,\ldots, k$.
\label{p.CSY}\end{theorem} 

Finally we can prove Theorem \ref{mmc2}.

\vspace{0,3cm}
{\it Proof of Theorem \ref{mmc2}:}

\vspace{0,1cm}
Let $\mathcal{R}$ be the residual subset given by Theorem \ref{p.CSY}. Hence, we consider $f\in \mathcal{R}$ and $p_0,\ldots, p_k$ the hyperbolic periodic points given by Theorem \ref{p.CSY}. Recall  $k=dim\ E^s\oplus E^c$, if $TM=E^s\oplus E^c\oplus E^u$ is the partially hyperbolic decomposition given by $f$. Since the index of a hyperbolic periodic point does not change for its analytic continuation, and since blender sets are robust, we can use Theorem \ref{rrtu} to find an open set $\mathcal{U}\subset \diff^1_m(M)$ arbitrary close to $f$ such that for every $g\in \mathcal{U}$ there exists a blender set $\Lambda_i(g)$ associated to each $p_i(g)$, for any $i=0,\ldots, k$.  

Now, using that generic volume preserving diffeomorphisms are transitive (see \cite{BC}) and the connecting lemma, we can find $g_1\in \mathcal{U}$ such that the blender set $\Lambda_1(g_1)$ is activated by $p_2(g_1)$. This implies, by   Lemma \ref{l.livro}, that there exists an open set $\mathcal{U}_1\subset \mathcal{U}$ such that the closure of $W^s(p_1(g))$ contains $W^s(p_2(g))$ for every $g\in \mathcal{U}_1$. Using the above arguments again, we can find an open set $\mathcal{U}_2\subset \mathcal{U}_1$ such that $W^s(p_2(g))$ contains $W^s(p_3(g))$ for every $g\in \mathcal{U}_1$. And thus, repeating this process finitely many times we can obtain an open set $\mathcal{V}\subset \mathcal{U}$ such that 

\begin{equation}
cl(W^s(p_i(g)))\supset W^s(p_{i+1}(g)), \quad \text{ for every } i=0,\ldots,k-1, \text{ and } g\in \mathcal{V}. 
\label{yes}\end{equation}

 Reducing the open set $\mathcal{V}$, if necessary, we can \cite{DW} and Remark \ref{argumentoACW} to assume that every $C^2-$diffeomorphism $g$ in $\mathcal{V}$ is weakly ergodic. 

Hence, let $g\in \mathcal{V}$ be a $C^2-$diffemorphism. Hence, $g$ is topologically transitive, which implies the existence of a dense backward orbit of $g$, say $\{g^{-n}(x)\}_{n\in \N}$. Since $g$ is partially hyperbolic,  the local strong unstable manifolds has uniform length, and thus there exists $n_0\in \N$ such that $\mathcal{F}^u(f^{-n_0}(x))$ intersects transversally $W_{loc}^s(p_{k}(g))$. Thus, the accumulation points of $\{f^{-n}(x)\}_{n\in \N}$ is also accumulated by points in $W^s(p_{k}(g))$, which implies the stable manifold of $p_k(g)$ is dense in $M$. Then, using (\ref{yes}) we conclude that $W^s(p_0(g))=\mathcal{F}^s(p_0(g))$ is also dense in $M$.

We have just proved  that $g$ satisfies the hypothesis of Theorem \ref{mmc2}, which implies $g$ is $ms$-minimal. 

Therefore, since $\mathcal{V}$ is arbitrary close to $f$, and $f$ is arbitrary in $\mathcal{R}$, by standard topology arguments we can find  an open set $\mathcal{A}_s\subset  \diff^1_m(M)\setminus cl(\mathcal{HT})$ such that any $g\in \mathcal{A}_s$  is $ms$-minimal, and $\mathcal{R}$ is contained in the closure of $\mathcal{A}_s$. 

Now, since the partially hyperbolic diffeomorphisms in $\mathcal{R}$ has non trivial extremal sub bundles, the above arguments can also be done for $f^{-1}$, to find an open set $\mathcal{A}_u\subset  \diff^1_m(M)\setminus cl(\mathcal{HT})$ such that any $g\in \mathcal{A}_u$  is $mu$-minimal, and $\mathcal{R}$ is also contained in the closure of $\mathcal{A}_u$. 

The proof of theorem is finished taking $\mathcal{A}=\mathcal{A}_s\cap \mathcal{A}_u$.

$\hfill\square$

\subsection{$m$-minimality in the symplectic setting:}

In this section we will prove Theorem \ref{symp.sett.}. Here $(M,\omega)$ is a symplectic manifold, being $\omega$ a symplectic form on $M$. Also, in this subsection $m$ will denotes the volume form on $M$ induced by the exterior powers of $\omega$.

The next result use m-elliptic periodic points to see density of strong leafs of partially hyperbolic diffemorphisms. 
Recall that a periodic point $p$ of a symplectic diffeomorphism $f$ with period $\tau(p)$ is called {\it m-elliptic} if $Df^{\tau(p)}(p)$ has exactly $2m$ modulus one eigenvalues, which must be non real and
simple eigenvalues.

Before we state the result, given $\delta$ and $\eps$ positive, we denote by $\mathcal{B}_{\omega}^{s}(\eps,\delta)$ and $\mathcal{B}_{\omega}^u(\eps,\delta)$  the set formed by symplectic partially hyperbolic diffeomorphisms $f$ such that  $m(\mathcal{X}_{\delta}^s(f))>1-\eps $ and 
$m(\mathcal{X}_{\delta}^u(f))>1-\eps$, respectively.

\begin{proposition}
Let $f\in \diff^1_{\omega}(M)$ be partially  hyperbolic having a $2m-$dimensional center bundle, and  a $m$-elliptic periodic point $p$. Thus, for any neighborhood $\mathcal{U}\subset \diff^1_{\omega}(M)$ of $f$ and any $\eps$ and $ \delta>0$ there exists a symplectic partially hyperbolic diffeomorphism $g\in \mathcal{U}$ which belongs to $\mathcal{B}^s(\eps,\delta)$ (resp. $\mathcal{B}^u(\eps,\delta)$).
\label{prop.symplectic}\end{proposition}

{\it Proof: }
We prove the existence of such $g$ inside $\mathcal{B}^s(\eps,\delta)$. The other case is a consequence by considering $f^{-1}$ instead of $f$. 

By continuity of the partially hyperbolic splitting and the robustness of $m$-elliptic periodic points in the symplectic scenario, we can suppose all diffeomorphisms in $\mathcal{U}$ are partially hyperbolic having a partially hyperbolic decomposition with same central bundle, and moreover, every diffeomorphism in $\mathcal{U}$ has a $m$-elliptic periodic point $p(g)$ which is the analytic continuation of $p$. 

After a perturbation, using Proposition \ref{prop.density}, we can assume $f$ is such that $\mathcal{F}^s(p)$ is dense in $M$. Now, let $\eps>0$ and $\delta>0$  given arbitrary. Since the strong stable foliation also varies continuously with the diffeomorphism in compact parts, taking a small neighborhood $V$ of $p$ and $\mathcal{U}$ smaller, if necessary, we can suppose $\mathcal{F}^s(x,g)$ is $\delta$-dense in $M$ for every diffeomorphism $g\in \mathcal{U}$ and $x\in V$. 

By Zender \cite{Ze}, there is a $C^2-$diffeomorphism $f_1\in \mathcal{U}$.   Moreover, as in the conservative setting, we can use accessibility and Remark \ref{argumentoACW} to assume that there exists a neighborhood $\mathcal{U}_1\subset \mathcal{U}$ such that every $C^2-$diffeomorphism $g$ in $\mathcal{U}_1$ is weakly ergodic.  
Recall that acessibility is also true in an open and dense subset among partially  hyperbolic symplectic diffeomorphisms.

To simplify the notation we still denote by $p$ the analytic continuation of $p$ for $f_1$.  Now, using pasting lemma, we can perturb $f_1$ to find a $C^2-$diffeomorphism $f_2\in \mathcal{U}_1$ such that $p$ still is a $m$-elliptic periodic point of $f_2$, and moreover $f_2=Df_1(p)$  in a small neighborhood of $p$, in local coordinates. Hence, replacing $V$ by a small neighborhood of $p$ and looking to $V$ in local coordinates, if we consider $E^s$ and $E^c$ the stable and center bundles of $f_2$, respectively, we have that  
 $D=(E^c\oplus E^u(p))\cap V$ is locally $f_2^{-\tau(p)}-$invariant. In fact, we have that $Df^{-1}_1(p)|E^u$ contracts and $Df^{-\tau(p)}_1(p)|E^c$ has norm equal to one. 

Now, since $f_2$ is $C^2$ and belongs to $\mathcal{U}_1$ it is weakly ergodic. Thus, for almost every point $x$ in $M$, $p$ belongs to $\omega(x)$, which implies by Proposition \ref{criterion.density} that $\mathcal{F}^s(x,f_2)$ is $\delta$-dense in $M$, since $\mathcal{F}^s(p,f_2)$ is. 

$\hfill\square$      

In the hypothesis of Proposition \ref{prop.symplectic} we have the existence of a $m$-elliptic periodic point for a partially hyperbolic diffeomorphism with center dimension equal to $2m$. This hypothesis was essential in the proof of such result. However, it is not directly that this fact happens for an arbitrary symplectic partially hyperbolic diffeomorphism. In fact, this was a question posed by \cite{ArBC}. Fortunately, it was proved in \cite{CH} that this happens for an open and dense subset among partially hyperbolic symplectic diffeomorphisms. More precisely:

\begin{theorem}[Theorem A in \cite{CH}]
\label{t.trichotomy}
There exists an open and dense subset $\mathcal{A}
\subset \diff^1_{\omega}(M)$, such that if $f\in \mathcal{A}$ is a partially hyperbolic diffeomorphism with  2m-dimensional center bundle, then $f$ has a $m$-elliptic periodic point. 
\end{theorem}

\vspace{0,2cm}
{\it Proof of Theorem \ref{symp.sett.}: } Let we consider the open and dense subset $\mathcal{A}$ inside the partially hyperbolic symplectic diffeomorphisms given by Theorem \ref{t.trichotomy}. Hence, given $m,n\in \N$, by Proposition \ref{prop.symplectic} we have that $\mathcal{B}^s(1/m,1/n)$ is dense in $\mathcal{A}$. Since these sets are open in $\diff^1_{\omega}(M)$ we have that $\mathcal{R}^s=(\cap \mathcal{B}^s(1/m,1/n)\cap \mathcal{A})\cup (cl(\mathcal{A}))^c$ is a residual subset inside  $\diff^1_{\omega}(M)$ such that every partially hyperbolic diffeomorphism $f\in \mathcal{R}^s$  is $ms$-minimal. 

Considering the maps $f^{-1}$, we can also find a residual subset $\mathcal{R}^u\subset \diff^1_{\omega}(M)$ such that every partially hyperbolic diffeomorphism $f\in \mathcal{R}^u$  is $mu$-minimal. 

Thus the proof is finished taking $\mathcal{R}=\mathcal{R}^s\cap \mathcal{R}^u$.  
$\hfill\square$

--------------------------------------------------------------------------


\begin{thebibliography}{00}

\bibitem[ABC]{ABC} F. Abdenur, C. Bonatti and S. Crovisier \emph{Nonuniform hyperbolicity for C1-generic diffeomorphisms.} Israel J. Math. 183 (2011), 1-60.



\bibitem[ArBC]{ArBC}  M-C. Arnaud, C. Bonatti, and S. Crovisier. \emph{Dynamiques symplectiques génériques},  Ergodic Theory and Dynamical Systems 25 (2005) 1401-1436, 2010.



\bibitem[ABCDW]{ABCDW} F. Abdenur, C. Bonatti, S. Crovisier, L. Diaz and L. Wen \emph{Periodic points and homoclinic classes.} 
Ergodic Theory Dynam. Systems 27 (2007), no. 1, 1-22.

\bibitem[AC]{AC} F. Abdenur and S. Crovisier \emph{Transitivity and topological mixing for C1 diffeomorphisms.} Essays in mathematics and its applications, 1-16, Springer, Heidelberg, 2012. 

\bibitem[ACS]{ACS} A. Arbieto, T. Catalan and B. Santiago \emph{Mixing-like properties for some generic and robust dynamics} Nonlinearity 28 (2015) 4103-4115.

\bibitem[ACW]{ACW} A. Avila, S. Crovisier and A. Wilkinson \emph{Diffeomorphisms with positive metric entropy} preprint 
 arXiv:1408.4252. 

\bibitem[BC]{BC} C. Bonatti and S. Crovisier \emph{Recurrence et generecite} Invent. Math. 158 (2004),  no. 1, 33-104.

\bibitem[BD]{BD}  C. Bonatti and L. Diaz. \emph{Persistent non-hyperbolic transitive diffeomorphisms,}
Annals of Math., 143 (2): 367-396, 1996

\bibitem[BDU]{BDU} C. Bonatti, L. Diaz and R. Ures \emph{Minimality of strong stable and unstable foliations for partially hyperbolic diffeomorphisms.}
J. Inst. Math. Jussieu 1 (2002), no. 4, 513-541.

\bibitem[BDV]{BDV} C. Bonatti, L. Diaz and M. Viana \emph{Dynamics beyond uniform hyperbolicity. A global geometric and probabilistic perspective.} Encyclopaedia of Mathematical Sciences, 102. Mathematical Physics, III. Springer-Verlag, Berlin, 2005

\bibitem[Bow1]{Bow} R. Bowen \emph{Equilibrium states and the ergodic theory of Anosov diffeomorphisms.} Second revised edition. With a preface by David Ruelle. Edited by Jean-Renï¿½ Chazottes. Lecture Notes in Mathematics, 470. Springer-Verlag, Berlin, 2008.

\bibitem[Bow2]{Bow1} R. Bowen \emph{Markov partitions for Axiom A diffeomorphisms.} Amer. J. Math. 92 1970 725-747. 

\bibitem[Ca]{Ca} T. Catalan, \emph{A $C^1$ generic condition for existence of symbolic extensions of volume preserving diffeomorphisms} Nonlinearity 25 (2012), no. 12, 3505-3525. 


\bibitem[CH]{CH}  T. Catalan and V. Horita \emph{$C^1$-Genericity of Symplectic Diffeomorphisms and Lower Bounds for Topological Entropy} (2013) preprint,  arXiv:1310.5162. 


\bibitem[CK]{CK} B. Carvalho and D. Kwietniak, \emph{On homeomorphisms with the two-sided limit shadowing property.}
J. Math. Anal. Appl. (2014) In press. http://dx.doi.org/10.1016/j.jmaa.2014.06.011



\bibitem[CMP]{CMP} C.M. Carballo, C.A. Morales, and M.J. Pac\'ifico, {\it Homoclinic classes for generic $C^1$ vector fields} Ergod. Th. and Dynam. Sys. (2003), 23, 403-415. 





\bibitem[CN]{CN} C. Camacho and A. Neto \emph{Geometric Theory of Foliations} Birkhäuser Boston, 1985. 


\bibitem[CSY]{CSY} S. Crovisier, M. Sambarino and D. Yang \emph{Partial hyperbolicity and homoclinic tangencies.} To appear at
J. Eur. Math. Soc. (JEMS)  arXiv:1103.0869v1 [math.DS]

\bibitem[DR]{DR} L. Diaz and J. Rocha \emph{Partial hyperbolicity and transitive dynamics generated by heteroclinic cycles.} Ergodic Theory and Dynamical Systems, 21:25-76, 2001. 

\bibitem[DW]{DW} D. Dolgopyat and A. Wilkinson \emph{ Stable accessibility is $C^1$ dense}  Astérisque No. 287 (2003), xvii, 33-60. 

\bibitem[F]{F} J. Franks \emph{Necessary conditions for stability of diffeomorphisms.}
Trans. Amer. Math. Soc. 158 1971 301-308.

\bibitem[FO]{FO} N. Friedman and D. Ornstein \emph{On isomorphism of weak Bernoulli transformations.} Advances in Math. 5 1970 365-394

\bibitem[H]{H} S. Hayashi \emph{Connecting invariant manifolds and the solution of the $C^1$ stability and $\Omega$-stability conjectures for flows.}
Ann. of Math. (2) 145 (1997), no. 1, 81-137.

\bibitem[Ha]{Ha} A. Hammerlindl \emph{Leaf conjugacies on the torus}  Ergodic Theory Dynam. Systems 33 (2013), no. 3, 896-933.

\bibitem[HHU]{HHU} F. Rodriguez Hertz, M. A. Rodriguez Hertz and R. Ures \emph{Some results on the integrability of the center bundle for partially hyperbolic diffeomorphisms.} Partially hyperbolic dynamics, laminations, and Teichmüller flow, 103-109, Fields Inst. Commun., 51, Amer. Math. Soc., Providence, RI, 2007


\bibitem[HPS]{HPS} M. Hirsch, C. Pugh and M. Shub \emph{Invariant manifolds.} Lecture Notes in Mathematics, Vol. 583. Springer-Verlag, Berlin-New York, 1977.

\bibitem[M]{M} R. Ma\~n\'e \emph{Contributions to the stability conjecture.}
Topology 17 (1978), no. 4, 383-396.

\bibitem[dMP]{dMP} W. de Melo and  J. Palis \emph{Geometric theory of dynamical systems.}
 Springer-Verlag, New York-Berlin, 1982

\bibitem[N]{N} S. E. Newhouse. \emph{ Quasi-elliptic periodic points in conservative dynamical systems}
 American Journal of Mathematics, 99(5) (1975), 1061-1087.

\bibitem[No]{No} F. Nobili \emph{ Minimality of invariant laminations for partially hyperbolic attractors}  Nonlinearity 28 (2015), no. 6, 1897-1918. 

\bibitem[P]{P}   Y. Pesin \emph{Nonuniform Hyperbolicity} Encyclopedia of Mathematics and its applications, Cambridge 2007.

\bibitem[PS]{PS} E.  Pujals and M. Sambarino \emph{
A sufficient condition for robustly minimal foliations} 
Ergodic Theory Dynam. Systems 26 (2006), no. 1, 281-289.
 
\bibitem[Shi]{Shi} K. Shinohara \emph{A note on minimality of foliations for partially hyperbolic diffeomorphisms} Arxiv 	arXiv:math/0602486.

\bibitem[S1]{S1} K. Sigmund \emph{On mixing measures for axiom A diffeomorphisms.} Proc. Amer. Math. Soc. 36 (1972), 497-504.

\bibitem[S2]{S2} K. Sigmund \emph{Generic properties of invariant measures for Axiom A diffeomorphisms.} Invent. Math. 11 1970 99-109.



\bibitem[Rob]{Rob} C. Robinson \title{Dynamical systems. Stability, symbolic dynamics, and chaos.} Studies in Advanced Mathematics. CRC Press, Boca Raton, FL, 1995. xii+468 pp. ISBN: 0-8493-8493-1 

\bibitem[RH]{RH} F. Rodriguez Hertz \emph{ Stable ergodicity of certain linear automorphisms of the torus}
Ann. of Math. (2) 162 (2005), no. 1, 65- 107. 

\bibitem[RHRHTU]{RHRHTU} F. Rodriguez Hertz, M.A. Rodriguez Hertz, A. Tahzibi, and R. Ures \emph{ Creation of blenders in the conservative setting}  Nonlinearity 23 (2010), no. 2, 211-223.


\bibitem[T]{T} A. Tahzibi \emph{ Robust transitivity and almost robust ergodicity}  Ergodic Theory Dynam. Systems 24 (2004), no. 4, 1261-1269. 

\bibitem[YGZ]{YGZ}  R. Yan, S. Gan and  P. Zhang \emph{An answer to Hammerlindl's question on strong unstable foliations} preprint  arXiv:1404.1702. 

\bibitem[Z]{Z} P. Zhang, \emph{ Partially hyperbolic sets with positive measure and ACIP for partially hyperbolic systems.} 
Discrete Contin. Dyn. Syst. 32 (2012), no. 4, 1435-1447. 


\bibitem[Ze]{Ze} E. Zehnder. \emph{Note on smoothing symplectic and volume-preserving diffeomorphisms.} Geometry and Topology (Proc. III Latin Amer. School of Math., Inst. Mat. Pura Aplicada CNPq, Rio de Janeiro, 1976) (Lecture Notes in Mathematics, 597). Springer, Berlin, 1977, pp. 828-854. MR0467846 (57 7697)

\end{thebibliography}
\end{document}